\author{Lennart Ronge}
\title{Hadamard expansions for powers of causal Green's operators and ``resolvents''}
\date{\null}
\documentclass[11pt]{article}
\usepackage{amsmath}
\usepackage{amssymb}
\usepackage{amsfonts}
\usepackage[ansinew]{inputenc}
\usepackage{amsthm}
\usepackage{geometry}
\usepackage[UKenglish]{babel}
\usepackage{bbm}
\usepackage{tikz}
\usepackage{graphicx}
\usepackage {wrapfig}
\usepackage{csquotes}
\usepackage[style=alphabetic, backend=bibtex]{biblatex}
\usepackage{url}
\usepackage{breakurl}

\def\<{\left<}
\def\>{\right>}

\def\Z{\mathbb{Z}}
\def\C{{\mathbb{C}}}
\def\R{{\mathbb{R}}}
\def\N{{\mathbb{N}}}

\def\insum#1{\sum\limits_{#1=0}^\infty}

\def\O {\backslash \{0\}}

\def\D{\mathcal{D}}
\def\K{\mathcal{K}}

\def\casedist#1#2#3#4{
	\begin{cases}
		#1& \text{, if } #2\\
		#3& \text{, if } #4 \\
\end{cases}}

\newtheorem{df}{Definition}[section]
\newtheorem*{dfn*}{Definition/Notation}
\newtheorem{lemma}[df]{Lemma}
\newtheorem{thm}[df]{Theorem}
\newtheorem*{thm*}{Theorem}
\newtheorem{prop}[df]{Proposition}
\newtheorem{cor}[df]{Corollary}

\newtheorem{gn}[df]{Fixed Notation}
\newtheorem{rem}[df]{Remark}
\newtheorem{defprop}[df]{Definition/Proposition}

\DeclareMathOperator{\Dom}{Dom}

\DeclareMathOperator{\supp}{supp}

\DeclareMathOperator{\Exp}{Exp}

\DeclareMathOperator{\grad}{grad}

\begin{document}
	\maketitle
	\begin{abstract}
		We derive an asymptotic expansion analogous to the Hadamard expansion for powers of advanced/retarded Green's operators associated to a normally hyperbolic operator $P$, as well as expansions for advanced/retarded Green's operators associated to $P-z$ for $z\in \C$. These expansions involve the same Hadamard coefficients as the original Hadamard expansion.
	\end{abstract}
	\section{Introduction}
	The Hadamard expansion is a formal expansion in differentiability orders of Green's operators associated to a wave operator. It takes the form 
	\[\K(G)\sim \insum{k}V_kR(2k+2),\]
	where $\K(G)$ denotes the Schwartz kernel of a Green's operator, the $V_k$ are sections in two variables known as the Hadamard coefficients and $R$ denotes a Lorentzian analogue to powers of the distance function. There are Hadamard expansions both for advanced/retarded and for Feynman Propagators, all involving the same Hadamard coefficients and different distributions $R$. A construction for both types of Hadamard expansions, showing both their similarities and differences, can be found e.g. in the appendix of \cite{BS}.
	
	The Hadamard expansions give a canonical way to split off a singular part from a Green's operator. This is used e.g. in the renormalization of the Stress-Energy Tensor on curved spacetimes (see e.g. \cite{DeFo}) and the Lorentzian version of the local APS index theorem (\cite{BS}). The Hadamard coefficients are also interesting in themselves, as they are the Lorentzian analogue for Riemannian heat kernel coefficients and encode geometric properties like the scalar curvature.
	
	In \cite{DaWr}, Dang and Wrochna develop Hadamard expansions for Feynman resolvents and their powers and use these to investigate a spectral action on Lorentzian scattering spaces. In \cite{DaWr2}, they use these expansions to investigate a generalized Wodzicki residue.
	
	Motivated by their work, we seek to develop similar expansions for powers of the advanced/retarded Green's operators. While the expansions we obtain are similar to those for the Feynman propagators in \cite{DaWr}, the methods used to derive them are quite different.
	
	In Section \ref{s2}, we introduce the setting and background material. In order to state the main results below and to have everything gathered in one place, we give a summary of the core definitions here: 
	
	\begin{dfn*}
		$P$ is a normally hyperbolic operator, acting on sections of a vector bundle $E$ over a globally hyperbolic manifold $M$. $U\subseteq M$ is an open, geodesically convex, globally hyperbolic and causally compatible subset. $\K(T)$ denotes the Schwartz kernel of an operator $T$. $G^\pm$ denotes the advanced/retarded Green's operator for $P$ (\ref{dfG}). $R_\pm$ denotes the Riesz distributions on $U$ (\ref{dfRiesz}). $V^k$ denotes the Hadamard coefficients (\ref{Vdef}) associated to $P$ on $U$. $\sim$ denotes an asymptotic expansion in differentiability orders (\ref{dfsim}).
	\end{dfn*}

	In Section \ref{s3}, we develop an expansion for powers of $G^\pm$ (see Theorem \ref{PowExp}) -- for $m\in \Z$, we have	
		\[\K(G^{\pm m})\sim\insum{k} \binom{m+k-1}{k}V^{k}R_\pm(2k+2m).\]
	
	In Section \ref{s4}, we develop asymptotic expansions for $G^\pm_{P-z}$, the advanced/retarded Green's operator for $P-z$, with $z\in\C$. We first show the formula for the Hadamard-coefficients of $P-z$,
	\[V^k(z)=\sum\limits_{m=0}^k\binom{k}{m}z^mV^{k-m},\]
	 which immediately gives a double-sum expansion in which the $z$-dependence is explicit (see Corollary \ref{doubleexp}):
		\[\K(G^{\pm}_{P-z})\sim\insum{k,m}\binom{k+m}{k}z^mV^{k}R_\pm(2k+2m+2).\]
	We then go on to define $z$-dependent analogues $R_\pm(z,2k)$ of the even Riesz distributions, with which we can write the expansion for $G^\pm_{P-z}$ in the same form as that for $G^\pm$ only with the $z$-dependent Riesz distributions instead of the usual ones (see Theorem \ref{Hadamexpz}):
		\[\K(G^{\pm}_{P-z})\sim \insum{k}V^{k}R_\pm(z,2k+2),\]
		or more generally
		\[\K(G^{\pm}_{P-z}\null^m)\sim \insum{k}\binom{k+m-1}{k}V^{k}R_\pm(z,2k+2m).\]

	The results of this paper are from the author's Ph. D. Thesis (\cite{Ron}), where they are somewhat separate from the other content. The aim of this paper is to provide the reader with an independent, shorter and more easily readable exposition of these results. For this reason, the setup and proofs that do not immediately concern the main results are kept more brief here, a more detailed description can be found in the original thesis.
	
	\subsection*{Acknowledgements}
	I would like to thank Matthias Lesch and Koen van den Dungen for their advice and support during the writing of both the thesis that this paper is based on and the paper itself.
	
	\section{Preliminaries}
	\label{s2}
	
	\subsection{Setting}
	Throughout this paper, we will assume that $M$ is a globally hyperbolic Lorentzian manifold with metric $g$ and $P$ is a normally hyperbolic operator, acting on sections of a vector bundle $E$ over $M$.	Normally hyperbolic means $P$ is a second order differential operator whose principal symbol is given by the minus the Lorentzian metric. Global hyperbolicity of $M$ is just needed to guarantee the existence and uniqueness of Green's operators.
	\begin{rem}
		As $M$ and $P$ are arbitrary, theorems proved and definitions made for $M$ or $P$ will hold for arbitrary globally hyperbolic manifolds and arbitrary normally hyperbolic operators. This will be exploited occasionally. Note that the restriction of a normally hyperbolic operator to an open subset is again normally hyperbolic.
	\end{rem}
\subsection{Notation}
We will often encounter definitions and statements that work in both directions of time symetrically, usually distinguished by an index $+$ or $-$. In order to avoid writing everything twice, we use the symbol $\pm$ to denote both cases at once. The use of $\pm$ or $\mp$ in a definition or theorem indicates that this should hold both with the upper sign chosen everywhere and with the lower sign chosen everywhere in that definition or theorem. In some cases we also need to change the wording depending on the choice of time alignment. In that case we will use a `/' to indicate that the first word is to be used with the upper sign and the second word is to be used with the lower sign.

For example, the statement `If $A$ is past/future compact, $J_\pm(A)\cap J_\mp(x)$ is compact.' means `If $A$ is past compact, $J_+(A)\cap J_-(x)$ is compact and if $A$ is future compact, $J_-(A)\cap J_+(x)$ is compact.'.

We write $J_\pm(x)$ for the causal future/past of a point or set $x$.
We write $\K(T)$ for the Schwartz kernel of an operator $T$. An operator defined on functions on $M$ that is applied to a function on $M\times M$ is taken to act in the first component, i.e.
\[(Tf)(y,x):=(T(f(\cdot,x)))(y).\]
More generally, an operator defined on distributions in $E$ is taken to act on distributions in $E\boxtimes F$ (for any vector bundle $F$) in the first component, i.e. for pure tensor products
\[T(f\otimes g):=(Tf)\otimes g.\]
This is defined so that for operators $T,S$ acting on sections of $E$, we have
\[\K(TS)=T\K(S).\]

\subsection{Spaces of sections}
In order to be able to define powers of Green's operators, we need to define them on the right domains. This requires us to talk about sections with past or future compact support.
\begin{df}
	A subset of $M$ is called past/future compact if and only if its intersection with the past/future of every point in $M$ is compact.
\end{df}
We shall need the following spaces of sections and distributions:
\begin{df}
	We denote by $\Gamma(E)$ the space of smooth sections and by $\Gamma^k(E)$ the space of k-times continuously differentiable sections of $E$. We denote by $\D'(E)$ the space of distributions in $E$. We use subscripts to indicate restrictions on the support. We write these for $\Gamma$, but also use the analogous notation with $\Gamma^k$ and $\D'$.
	\begin{itemize}
		\item For $A\subseteq M$, $\Gamma_A(E)$ denotes sections supported in $A$.
		\item $\Gamma_c(E)$ denotes compactly supported sections.
		\item $\Gamma_\pm (E)$ denotes sections with past/future compact support.
	\end{itemize}
\end{df}
All these spaces carry canonical topologies such that all inclusions are continuous and $\Gamma_c(E)$ is dense in all spaces. For example, $\Gamma_A(E)$ carries the topology induced by all $C^k$-norms on compact subsets and $\Gamma_\pm(E)$ carries the limit topology of all $\Gamma_A(E)$ with $A$ past/future compact. See \cite[section 2]{GH} for a more detailed description of these spaces.
\subsection{Green's operators}
With this setup, the advanced and retarded Green's operators can be defined as inverses of $P$ on the appropriate spaces.
\begin{defprop}
	\label{dfG}
	$P$ is invertible as a map of $\D'_\pm(E)$ to itself. The Green's operator $G^\pm_P$ is defined to be the inverse of $P$ on $\D'_\pm(E)$. This is continuous and restricts to a continuous map of $\Gamma_\pm (E)$ to itself. We omit the subscript $P$, if it is clear what operator we are talking about.
	For any $f\in \D'_\pm(E)$, we have $\supp(G^\pm f)\subseteq J_\pm(\supp(f))$.
\end{defprop}
This follows from \cite[Theorem 3.8, Corollary 3.11 and Lemma 4.1]{GH} (that $P$ is Green's hyperbolic and thus the previous results apply, holds e.g. by \cite[Corollary 3.4.3]{BGP}). It is more customary to define Green's operators as maps $\Gamma_c(E)\rightarrow \Gamma(E)$. This is equivalent to the definition used here via restriction/continuous extension. The definition we use has the advantage that (integer) powers of Green's operators are immediately well defined, as the Green's operators are automorphisms.

\subsection{Suitable domains}
Most considerations in this paper will not take place on all of $M$, but on suitable subsets of $M$:
\begin{df}
	\label{dfGE}
	A subset $U$ of $M$ shall be called GE, if it is open, geodesically convex, globally hyperbolic, and causally compatible. 
\end{df}
\begin{gn}
	From now on, fix an arbitrary GE set $U$ in $M$.
\end{gn}
Openness is required to restrict distributions to $U$. Geodesic convexity will be required to define Hadamard coefficients and Riesz distributions. Global hyperbolicity ensures that $U$ has unique Green's operators, while causal compatibility ensures that those agree with the restriction of the Green's operators on $M$. Indeed, by \cite[Proposition 3.5.1.]{BGP} (and continuity), we have for any $f\in \D'_\pm(E)$:
	\[G^\pm_{P|_U}(f|_U)=(G^\pm_{P}f)|_U.\]
In the following, we will not distinguish between $G^\pm_{P}$ and $G^\pm_{P|_U}$ notationally, as the previous equality guarantees that this does not lead to ambiguity. Moreover, we will often omit the restriction of kernels to $U\times U$ from the notation.

\begin{rem}
	Again, any theorem and definition made in reference to $U$ works for arbitrary GE subsets of globally hyperbolic manifolds. This includes Lorentzian vector spaces as GE subsets of themselves.
\end{rem}
The notion of GE sets is also sufficiently general to investigate local properties around any point: \cite[Corollary 2 and Remark 14]{Min} imply that every point in $M$ has a neighborhood basis of GE-sets.
\subsection{Riesz distributions and Hadamard coefficients}
\label{Rieszhadamard}
In this section, we introduce Riesz distributions and Hadamard coefficients, which will play a central role in this paper. We will follow \cite{BGP} here, they are also described in \cite {Gun} and \cite {Fri}.
As all constructions rely heavily on the exponential map, the objects will only be defined on convex subsets of a Lorentzian manifold. We will define them for $U$ and omit the index $U$ after the definition, unless we explicitly want to use them for a different set.

The Riesz distributions on $U$ are roughly speaking a Lorentzian analogue to `powers of the distance function'. They are described in detail in \cite[sections 1.2 and 1.4]{BGP}. Here, they  will play a role comparable to that of powers of $x$ in a Taylor series.
\begin{df}
	\label{dfgamma}
	For $x,y\in U$, define
	\[\Gamma^U(y,x):=-g(exp_x^{-1}(y),exp_x^{-1}(y)).\]
\end{df}
This is the Lorentzian analogue of the Riemannian distance squared. $\Gamma(y,x)$ is positive if and only if $y$ is in the future or past of $x$. We now define the Riesz distributions (see \cite[Definitions 1.2.1 and 1.4.1, Lemma 1.2.2 and Proposition 1.4.2]{BGP} for the definition and the following remarks)
\begin{defprop}
	\label{dfRiesz}
	For $\alpha\in \C$ with real part $\Re(\alpha)>d$ and $x\in U$, define Riesz distributions $R^{U}_\pm(\alpha)(\cdot,x)$ as the distribution on $U$ given by the function
	\[R^U_\pm(\alpha)(y,x)=\casedist{c_\alpha\Gamma(y,x)^{\frac{\alpha-d}{2}}}{x\in J_\pm(0)}{0}{x\notin J_\pm(0)}\]
	with
	\[c_\alpha:=\frac{2^{1-\alpha}\pi^\frac{2-d}{2}}{\Gamma(\frac{a}{2})\Gamma(\frac{\alpha-d+2}{2})}.\]
	The map $\alpha\mapsto R^{U}_\pm(\alpha)(\cdot,x)$ is holomorphic as a map into $\D'(U)$ and extends uniquely to a holomorphic map on all of $\C$. For arbitrary $\alpha\in\C$, define $R^{U}_\pm(\alpha)(\cdot,x)$ to be the value of this holomorphic extension. 
	This defines Riesz distributions $R^{U}_\pm(\alpha)$ on $U\times U$ as suggested by the notation:
	\[R^{U}_\pm(\alpha)[\psi]:=\int\limits_UR^{U}_\pm(\alpha)(\cdot,x)[\psi(\cdot,x)]dx\]
	for any test function $\psi\in C_c^\infty(U\times U)$.
\end{defprop}
The $\Gamma$ in the definition of $c_\alpha$ refers to the gamma function, not the $\Gamma$ defined above. The prefactor is chosen such that 
\[\square R^{\R^d}_\pm(\alpha+2)=R^{\R^d}_\pm(\alpha).\]
 The precise form of $c_\alpha$ will not be relevant in the following and there will be no further occurences of gamma functions. We have
 \[R_\pm(0)=\delta,\]
 i.e. the 0-th Riesz distribution is the Dirac distribution on the diagonal.
The Riesz distributions on $U$ can be obtained from those on Minkowski space via the exponential map:
\[R^{U}_\pm(\alpha)(\cdot,x)=(exp_x)^{-1*}R^{T_xM}_\pm(\alpha)(\cdot,0).\]
They are supported in the causal past/future of the basepoint:
\[\supp(R^{U}_\pm(\alpha)(\cdot,x))\subseteq J_\pm(x).\]
As $\Gamma$ vanishes at the boundary of the causal past and future, looking at the formula for $\Re(\alpha)>d$ shows that $R_\pm(\alpha)$ is $C^n$ if $\Re(\alpha)>d+2n$.

On Minkowski space, the Riesz distributions $R^{\R^d}_\pm(2k)$  for $k\in \N$ can be thought of as fundamental solutions to $\square^k$:
\begin{prop}
	\label{RFS}
	On a Lorentzian vector space $V$, we have
	\[ R_\pm^V(2m)=G^{\pm m}_{\square}\delta,\]
	where $\delta$ denotes the dirac distribution on the diagonal and $\square$ the d'Alembertian on $V$.
\end{prop}
\begin{proof}
	We have
	\[R_\pm^V(0)=\delta\]
	and
	\[\square_V R_\pm^V(2m+2)=R_\pm^V(2m).\]
	Thus
	\[R_\pm^V(2m+2)=G^\pm_{\square_V}R_\pm^V(2m).\]
	Iterating this, we obtain
	\[ R_\pm^V(2m)=G^{\pm m}_{\square_V}R_\pm^V(0)=G^{\pm m}_{\square_V}\delta.\qedhere\]
\end{proof}

We now turn to the second class of objects we want to define in this section: the Hadamard coefficients.
In order to define them, we introduce the following differential operator.
\begin{df}
	\label{dfrho}
	Define for $V\in \Gamma(E\boxtimes E^*)$:
	\[\rho^U V(y,x):= \nabla_{\grad\Gamma(y,x)}V(y,x)-\left(\frac{1}{2}\square\Gamma(y,x)-d\right)V(y,x),\]
	where $\nabla$ is the connection induced by $P$.
	
\end{df}
The reader may immediately forget the definition of $\rho$ and only remember the following property:
\begin{prop}
	\label{Prho}
	We have for $\alpha\in \C\O$, $x\in U$ and $V\in \Gamma(E\boxtimes E^*|_{U\times U})$
	\[ P(VR_\pm(\alpha+2))=(PV)R_\pm(\alpha+2)-\frac{1}{\alpha}((\rho-\alpha)V)R_\pm(\alpha)\]
\end{prop}
\begin{proof}
	By \cite[Proposition 1.4.2]{BGP}, the Riesz distributions satisfy
	\[2\alpha\grad(R_\pm(\alpha+2))=(\grad\Gamma) R_\pm(\alpha)\]
	and
	\[\alpha\square R_\pm(\alpha+2)=\left(\frac{1}{2}\square\Gamma-d+\alpha\right) R_\pm(\alpha).\]
	Using this and the Leibniz rule for normally hyperbolic operators, we calculate for $\Re(\alpha)>d+4$ (so $R_\pm(\alpha,x)$ is $C^2$ and we don't have to worry about distributions)
	\begin{align*}
		&\alpha P (VR_\pm(\alpha+2,x))\\
		=&\alpha\left((P V)R_\pm(\alpha+2,x)-2\nabla_{\grad R_\pm(\alpha+2,x)} V+V\square R_\pm(\alpha+2,x)\right)\\
		=&\alpha(P V)R_\pm(\alpha+2,x)-(\nabla_{\grad\Gamma} V) R_\pm(\alpha,x)+ V\left(\frac{1}{2}\square\Gamma_x-d+\alpha\right)R_\pm(\alpha,x)\\
		=&\alpha(P V)R_\pm(\alpha+2,x)-((\rho_x^U-\alpha)V)R_\pm(\alpha,x).
	\end{align*}
	As both sides are holomorphic, the equation holds for arbitrary $\alpha$.
\end{proof}
 Pretty much exactly this calculation is done in \cite[equation 2.1]{BGP} to motivate the definition of the Hadamard coefficients, which we will give now:
\begin{defprop}
	\label{Vdef}
	There is a unique family of smooth sections \[V^{k,U}\in \Gamma(E\boxtimes E^*|_{U\times U})\] (indexed by $k\in \N$) that satisfies the transport equations 
	\[(\rho-2k) V^{k,U}=2kP V^{k-1,U}\]
	(for $k=0$, the right hand side is set to $0$) subject to the initial condition
	\[V^{0,U}(x,x)=1.\]
	We call these the Hadamard coefficients.
\end{defprop}
The motivation for defining the Hadamard coefficients is that we want to have the formal equality
\[P\insum{k}V^{k}R_\pm(2k+2)=\delta,\]
which (together with support conditions) means that the infinite sum, if it existed, would be the Schwartz kernel of $G^\pm$. In general, the sum does not exist and we only get a form of asymptotic expansion.
\subsection{Asymptotic expansions in differentiability orders}
The Hadamard expansion, as well as the expansions developed in this paper, are not exact equalities. Rather, the quantity on the left hand side can be approximated up to arbitrary degree of differentiability by partial sums of the right hand side:
\begin{df}
	\label{dfsim}
	For distributions $F$ and $f_k$, we write
	\[F\sim\insum{k}f_k\]
	if and only if for any $m\in \N$ there is $N_0\in \N$ such that for all $N>N_0$, the difference
	\[F-\sum\limits_{n=0}^Nf_n\]
	is $m$ times continuously differentiable.
\end{df}
If we have multiple sums on the right hand side, we define this analogously (as one would do with limits).
\begin{rem}
	As both Green's kernels and Riesz distributions vanish outside the causal past/future of the second argument, a $C^k$-remainder in the following expansions also vanishes of order $k$ near the boundary of the past/future.
\end{rem}

	\section{Powers of Green's operators}
	\label{s3}
	We first derive an expansion for powers of the advanced/retarded Green's operators. Before we can start the main proof, we need to do some technical work to guarantee that the Green's operators do not decrease differentiability orders too much (locally). This will be necessary to control the remainder terms of our asymptotic expansion. The result is basically a consequence of mapping smooth sections to smooth sections continuously.
	\begin{prop}
		Let $A\subset M$ be past/future compact compact and $W\subset M$ be open and relatively compact. Let $r_W$ denote restriction to $W$. Then for every $n\in \N$ there is $m\in \N$ such that $r_W\circ G^\pm$ maps $\Gamma^m_A(E)$ to $\Gamma^n(E|_W)$ continuously.
	\end{prop}
	\begin{proof}
		We know  that
		\[G^\pm\colon \Gamma_\pm(E)\rightarrow \Gamma_\pm(E)\]
		is continuous. Since the topologies of $\Gamma_{A}(M)$ and $\Gamma_{J_\pm(A)}(M)$ coincide with the subspace topology from $\Gamma_\pm$ (as the latter carries the limit topology), we know that the restriction
		\[G^\pm\colon \Gamma_{A}(E)\rightarrow \Gamma_{J_\pm(A)}(E)\]
		is also continuous. $C^k$- norms on ${\overline{W}}$ are seminorms on the target space. Thus by the seminorm-characterization of continuity, we know that for any $n\in \N$, there is $c\in\R$, $m\in \N$ and $K\subseteq M$  such that for all $\phi\in \Gamma_A(E)$
		\[\|G^\pm\phi\|_{C^n(\overline{W})}\leq c\|\phi\|_{C^m(K)}.\]
		The right hand side is also a seminorm of $\Gamma^m_A(E)$.
		Thus $r_W\circ G^\pm|_{\Gamma_{A}(E)}$ extends continuously to a map $\Gamma^m_A(E)\rightarrow \Gamma^n(E|_W)$. As $C^k$- convergence implies distributional convergence, this continuous extension must coincide with $r_W\circ G^\pm|_{\Gamma^m_{A}(E)}$, as the latter is distributionally continuous.
	\end{proof}
	As Hadamard coefficients and Green's kernels are sections in $E\boxtimes E^*$ and we want differentiability in both coordinates, we need to extend the above to the case where $G^\pm$ acts in the first component on sections of a box product.
	\begin{prop}
		\label{Greg}
		Let $A\subset M$ be past/future compact and $W\subset M$ be open and relatively compact. Let $O$ be some manifold and $F$ some vector bundle over $O$. Then for each $n\in \N$, there is $m\in \N$ such that for any $f\in \Gamma^m(E\boxtimes F)$ supported in $A\times O$, we have $G^\pm f|_{W\times O}\in \Gamma^n((E\boxtimes F)|_{W\times O})$, where $G^\pm f(y,x):=G^\pm (f(\cdot,x))(y)$.
	\end{prop}
	\begin{proof}
		As differentiability can be checked locally, i.e. in charts on which $E\boxtimes F$ is trivial, it suffices that the partial derivatives up to order $m$ of the component functions in these charts exist. Choose $m$ such that $G^\pm$ maps $\Gamma^{m-n}_{A}(E)$ to $\Gamma^n(E)$ continuously. As $\partial^I\circ G^\pm$ is continuous as a map into $C^{n-|I|}$-sections for any multiindex with $|I|<n$, the differential quotients for partial derivatives in the second component can be pulled inside the operator, where the limits are well-defined and give functions of one less differentiability order. Iterating this, starting with a $C^m$-section and moving all partial derivatives in the second coordinates inside, the argument stays at least $C^{m-n}$, so we obtain that all partial derivatives up to order $n$ of $G^\pm f$ are well defined and continuous.
	\end{proof}

	We can now prove the main theorem of this section:
	\begin{thm}
		\label{PowExp}
		Let $U\subseteq M$ be a GE set and let $m\in \Z$. Then we have the asymptotic expansion
		\[\K(G^{\pm m})|_{U\times U}\sim\insum{k} \binom{m+k-1}{k}V^{k,U}R^U_\pm(2k+2m).\]
		For $m<0$, the right hand side is a finite sum, as all summands with $k+m>0$ vanish, and we have equality.
	\end{thm}
	\begin{rem}
		Note that the special case $m=1$ reproduces the original Hadamard expansion.
	\end{rem}
	\begin{rem}
		In case $m<0$, this is actually a formula for the kernel of  powers of $P$. All even Riesz distributions of order $0$ or less are the same in the advanced and retarded case and supported only at the diagonal, so this expansion is not as surprising as it might seem on first glance. From the case $m=-1$, we obtain that 
		\[\K(P)=V^0R_\pm(-2)-V^1\delta.\]
		In particular, $P$ is uniquely determined by the metric $g$ (which determines the Riesz-distributions) and the first two Hadamard coefficients $V^0$ and $V^1$.
	\end{rem}
	\begin{proof}
		We proceed by two-way induction on $m$, showing that if the statement holds for $m\in \Z$, it also holds for $m+1$ if $m\geq 0$ and for $m-1$, if $m\leq 0$. The case $m=0$ follows from $R_\pm(0)=\delta$, as all summands excapt the first vanish.
		
		We make some preliminary calculations. For $k+m\neq 0$ we can use the characterization of $\rho$ (\ref{Prho}) and the transport equations (see \ref{Vdef}) to calculate
		\begin{align*}
		&\binom{k+m}{k}P(V^{k}R_\pm(2k+2m+2))\\
		&=\binom{k+m}{k}\Big(\frac{-1}{2k+2m}(\rho-2k-2m)V^{k}R_\pm(2k+2m)+(PV^{k})R_\pm(2k+2m+2)\Big)\\
		&=\binom{k+m}{k}\Big(\frac{-1}{2k+2m}((2kPV^{k-1}-2mV^{k})R_\pm(2k+2m))\\
		&+(PV^{k})R_\pm(2k+2m+2)\Big)\\
		&=\binom{k+m}{k}\frac{2m}{2k+2m}V^{k}R_\pm(2k+2m)-\binom{k+m}{k}\frac{2k}{2m+2k} (PV^{k-1})R_\pm(2k+2m)\\
		&+\binom{k+m}{k}(PV^{k})R_\pm(2k+2m+2)\\
		&=\binom{k+m-1}{k}V^{k}R_\pm(2k+2m)-\binom{k+m-1}{k-1} (PV^{k-1})R_\pm(2k+2m)\\
		&+\binom{k+m}{k}(PV^{k})R_\pm(2k+2m+2))\\
		\end{align*}
		In the special case $k=m=0$, we obtain:
		\begin{align*}
		P(V^{0}R_\pm(2))&=\lim\limits_{\alpha\rightarrow 0}P(V^{0}R_\pm(\alpha+2))\\
		&=\lim\limits_{\alpha\rightarrow 0}P(V^{0})R_\pm(\alpha+2)-(\frac{1}{\alpha}\rho V^{0}-V^{0})R_\pm(\alpha)\\
		&=\lim\limits_{\alpha\rightarrow 0}P(V^{0})R_\pm(\alpha+2)+ V^{0}R_\pm(\alpha)\\	
		&=P(V^{0})R_\pm(2)+ V^{0}R_\pm(0)\\
		&=P(V^{0})R_\pm(2)+ V^{0}\delta\\
		&=\delta+P(V^{0})R_\pm(2),
		\end{align*}
		where we used that $V_0$ is $1$ on the diagonal.
		Either way, we obtain for $k=0$:
		\[P(V^{0}R_\pm(2m+2))=V^{0}R_\pm(2m)+(PV^{0})R_\pm(2m+2)\]
		
		Putting everything together, for any $N\in\N$ in case $m\geq 0$ and for any $N<-m$ if $m<0$, we have
		\begin{align*}
		&P\sum\limits_{k=0}^N\binom{m+k}{k}V^{k}R_\pm(2k+2m+2)\\
		&=V^{0}R_\pm(2m)+(PV^{0})R_\pm(2m+2)+\sum\limits_{k=1}^N\binom{k+m-1}{k}V^{k}R_\pm(2k+2m)\\
		&-\binom{k+m-1}{k-1} (PV^{k-1})R_\pm(2k+2m) +\binom{k+m}{k}(PV^{k})R_\pm(2k+2m+2)\\
		&=\sum\limits_{k=0}^N\binom{k+m-1}{k}V^{k}R_\pm(2k+2m)-\sum\limits_{k=0}^{N-1}\binom{k+m}{k}(PV^{k})R_\pm(2k+2m+2)\\
		&+\sum\limits_{k=0}^{N}\binom{k+m}{k}(PV^{k})R_\pm(2k+2m+2)\\
		&=\sum\limits_{k=0}^N\binom{k+m-1}{k}V^{k}R_\pm(2k+2m)+\binom{N+m}{N}(PV^{N})R_\pm(2N+2m+2)\\
		&=:\sum\limits_{k=0}^N\binom{k+m-1}{k}V^{k}R_\pm(2k+2m)+E_N\\
		\end{align*}
		where 
		\[E_N:=\binom{N+m}{N}(PV^{N})R_\pm(2N+2m+2)\]
		is $C^{n}$ for $N\geq \frac{d}{2}+n-m$.
		
		\textbf{Case 1: $m+1\Rightarrow m$, $m<0$}\\
		Suppose for induction that
		\[\K(G^{\pm m+1})=\sum\limits_{k=0}^{-m-1}\binom{m+k}{k}V^{k}R_\pm(2k+2m+2)\]
		(all further summands are 0). Let $N=-m-1$. 
		We use the fact (see \cite[below Proposition 2.3.1]{BGP}) that on the diagonal $\Delta=\{(x,x)|x\in U\}$, we have
		\[PV^{N}|_{\Delta}=-V^{N+1}|_\Delta\]
		and the binomial identity
		\[\binom{-a}{b}=(-1)^b\binom{a+b-1}{b}\] 
		for integer $b$, which implies in particular
		\[\binom{-1}{b}=(-1)^b.\]
		 With these, we obtain
		\begin{align*}
		E_N&=\binom{-1}{N}(PV^{N})R_\pm(0)\\
		&=(-1)^N(PV^{N})\delta\\
		&=(-1)^{N+1}V^{N+1}\delta\\
		&=\binom{-1}{N+1}V^{N+1}R_\pm(0)\\
		&=\binom{-1}{-m}V^{-m}R_\pm(0).
		\end{align*}
		We obtain from the inductive hypothesis and our previous calculations:
		\begin{align*}
		\K(G^{\pm m})
		&=P\K(G^{\pm m+1})\\
		&=P\sum\limits_{k=0}^{N}\binom{m+k}{k}V^{k}R_\pm(2k+2m+2)\\
		&=\sum\limits_{k=0}^N\binom{k+m-1}{k}V^{k}R_\pm(2k+2m)+E_N\\
		&=\sum\limits_{k=0}^{-m-1}\binom{k+m-1}{k}V^{k}R_\pm(2k+2m)+\binom{-1}{-m}V^{-m}R_\pm(0)\\
		&=\sum\limits_{k=0}^{-m}\binom{k+m-1}{k}V^{k}R_\pm(2k+2m)\\
		\end{align*}
		
		\textbf{Case 2: $m\Rightarrow m+1$, $m\geq 0$}\\
		Fix $n\in \N$ and assume the theorem holds for some $m\geq 0$. We first show that the expansion holds on relatively compact GE subsets of $U$. Let $W,O\subseteq U$ be relatively compact (in $U$) and GE. Proposition \ref{Greg} implies that there is $n'\in \N$ such that $G^\pm$ maps $C^{n'}$-sections supported in  $J^U_\pm(\overline O)\times O$ to sections that are $C^n$ on $W\times O$. By the inductive hypothesis, we may choose $N\geq \frac{d}{2}+n'$ such that
		\[F_N:=G^{\pm m}-\sum\limits_{k=0}^N \binom{m+k-1}{k}V^{k}R_\pm(2k+2m)\]
		is a $C^{n'}$-section. We then have on $U\times U$
		\begin{align*}
		&\K(G^{\pm m+1})\\
		&=G^\pm\K(G^{\pm m})\\
		&=G^\pm\left(\sum\limits_{k=0}^N \binom{m+k-1}{k}V^{k}R_\pm(2k+2m)+F_N\right)\\
		&=G^\pm\left(P\sum\limits_{k=0}^N\binom{m+k}{k}V^{k}R_\pm(2k+2m+2)-E_N+F_N\right)\\
		&=\sum\limits_{k=0}^N\binom{m+k}{k}V^{k}R_\pm(2k+2m+2)+G^\pm(F_N-E_N)\\
		&=:\sum\limits_{k=0}^N\binom{m+k}{k}V^{k}R_\pm(2k+2m+2)+\tilde F_N.\\
		\end{align*}
		with
		\[\tilde F_N:=\K(G^{\pm m+1})-\sum\limits_{k=0}^N\binom{m+k}{k}V^{k}R_\pm(2k+2m+2)=G^\pm(F_N-E_N).\]
		By our choice of ${n'}$ and $N$, $\tilde F_N$ is $C^n$ on $W\times O$, as both $E_N|_{U\times O}$ and $F_N|_{U\times O}$ are $C^{n'}$ and supported in $J_\pm(O)\times O$. Note that the only choices we made after choosing $W$ and $O$ were that of ${n'}$ and $N$. We thus want to make the expansion independent of the choice of $N$.
		
		Let $N'\geq\frac{d}{2}+n$ be arbitrary. For $k\geq N'$, all Riesz distributions $R_\pm(2k+2m+2)$ are given by $C^n$-functions. We see that
		\begin{align*} 
		\tilde F_{N'}&=\K(G^{\pm m+1})-\sum\limits_{k=0}^{N'}\binom{m+k}{k}V^{k}R_\pm(2k+2m+2)\\
		&=\sum\limits_{k=N'+1}^{N}\binom{m+k}{k}V^{k}R_\pm(2k+2m+2)+\tilde F_N
		\end{align*}
		is $C^n$ on $W\times O$. As this is independent of $W$ and $O$, these can replaced by arbitrary relatively compact GE sets. All points in $U$ have relatively compact GE neighborhoods (choose any relatively compact neighborhood and then choose a GE neighborhood contained in it). Thus any $(y,x)\in U\times U$ has a neighborhood on which $\tilde F_{N'}$ is $C^n$, which means it is $C^n$ on all of $U\times U$. As $n$ was arbitrary, we have the desired asymptotic expansion for $m+1$ and have thus concluded our induction.
	\end{proof}

	\section{Green's ``resolvent''}
	\label{s4}
	We now consider the Green's operators for the operator $P-z$ for $z\in \C$, which is still normally hyperbolic. If we view $G^\pm_P$ as something like an inverse for $P$, then $G^\pm_{P-z}$ is something like a resolvent. We seek to derive an expansion similar to the one above, where only the Hadamard coefficients of $P$ rather than $P-z$ appear and the $z$-dependence is instead shifted to a generalzed version of Riesz distributions.

	We want to exploit the asymptotic expansion that we already have in order to obtain the one we would like to get. For this, we express the Hadamard coefficients for $P-z$ in terms of those for $P$.
	
	\begin{df}
		\label{defVz}
		For $z\in \C$,  let $V^{k}(z)$ denote the Hadamard coefficients for $P-z$.
	\end{df}
	\begin{prop}
		\label{Vz}
		We have
		\[V^{k}(z)=\sum\limits_{m=0}^k\binom{k}{m}z^mV^{k-m}.\]
	\end{prop}
	\begin{rem}
		The corresponding formula for the heat coefficients of a Laplace operator,
		\[a_k(\Delta-z)=\sum\limits_{m=0}^k\frac{1}{m!}z^ma_{k-m}(\Delta),\]
		can be shown analogously. Alternatively, that formula can also be deduced by Taylor expanding $e^{tz}$ and multiplying out the expansions in 
		\[e^{-t(\Delta-z)}=e^{tz}e^{-t\Delta}.\]
	\end{rem}
	\begin{proof}
		Let
		\[V_k(z):=\sum\limits_{m=0}^k\binom{k}{m}z^mV^{k-m}\]
		We need to show that $V_0(z)(x,x)=1$ and the $V_k(z)$ satisfy the transport equations.
		The former holds, as $V_0(z)=V^0$. For the latter, we calculate using the transport equation for $P$:
		\begin{align*}
		&(\rho-2k)V_k(z)\\
		&=\sum\limits_{m=0}^k\binom{k}{m}z^m(\rho-2k)V^{k-m}\\
		&=\sum\limits_{m=0}^k\binom{k}{m}z^m((\rho-2(k-m))V^{k-m}-2mV^{k-m})\\
		&=\sum\limits_{m=0}^k\binom{k}{m}z^m2(k-m)PV^{k-m-1}-\sum\limits_{m=0}^k\binom{k}{m}z^m2mV^{k-m}\\
		&=\sum\limits_{m=0}^k\binom{k-1}{m}z^m2kPV^{k-m-1}-\sum\limits_{m=1}^k\binom{k-1}{m-1}z^m2kV^{k-m}\\
		&=\sum\limits_{m=0}^{k-1}\binom{k-1}{m}z^m2kPV^{k-m-1}-\sum\limits_{m=0}^{k-1}\binom{k-1}{m}z^{m+1}2kV^{k-m-1}\\
		&=2k(P-z)\sum\limits_{m=0}^{k-1}\binom{k-1}{m}z^mV^{k-m-1}\\
		&=2k(P-z)V_{k-1}(z).
		\end{align*}
		The $V_k(z)$ thus satisfy the transport equations for $P-z$, so they are the desired Hadamard coefficients.
	\end{proof}
	Inserting this into the Hadamard expansion, we obtain our first expansion for the Green's kernel of $P-z$, an asymptotic expansion in a double sum, where the $z$ dependence is explicit:
	\begin{cor}
		\label{doubleexp}
		For every GE set $U\subseteq M$, we have the asymptotic expansion
		\[\K(G^{\pm}_{P-z})|_{U\times U}\sim\insum{k,m}\binom{k+m}{k}z^mV^{k,U}R^U_\pm(2k+2m+2).\]
	\end{cor}
	\begin{proof}
		Taking the standard Hadamard expansion to $P-z$, inserting our formula for the $z$-dependent Hadamard coefficients (Proposition \ref{Vz}) and re-indexing the summands, we obtain
		\begin{align*}
		\K(G^{\pm}_{P-z})&\sim \insum{n}V^{n}(z)R_\pm(2n+2)\\
		&=\insum{n}\ \sum\limits_{k+m=n}\binom{n}{k}z^mV^{k}R_\pm(2n+2)\\
		&=\insum{k,m}\binom{k+m}{k}z^mV^{k}R_\pm(2k+2m+2)\qedhere
		\end{align*}
	\end{proof}
	This gives us an explicit expression for the $z$-dependence. We now seek to  shift the $z$-dependence into the Riesz distributions to obtain something closer to the classical Hadamard expansion. The even Riesz distributions are obtained from shifting fundamental solutions of the d'Alembert operator on Minkowski space to the manifold 
	\[R^{U}_\pm(2m)(\cdot,x)=(exp_x)^{-1*}(G^{\pm m}_{\square}\delta_0)\]
	(see Proposition \ref{RFS} and the preceding remarks). A reasonable $z$-dependent generalization is to do the same for $\square-z$ instead:
		\begin{df}
		\label{Rieszres}		
		For $z\in \C$, $x\in M$ and $m \in \N$, we define resolvent Riesz distributions as
		\[R_\pm(z,2m)(\cdot,x):=(exp_x^{-1})^*(G^{\pm m}_{\square-z}\delta_0),\]
		and the versions on $U\times U$
		\[R_\pm(z,2m)[\psi]=\int\limits_U R_\pm(z,2m)(\cdot,x)[\psi(\cdot,x)]dx.\]
		Here $\square$ is the d'Alembertian on $T_x M$.
	\end{df}
	
	To construct our final expansion, we need an expansion for the resolvent Riesz distributions in terms of the standard ones. We start by doing this in Minkowski space:
	\begin{lemma}
		On any Lorentzian vector space $V$, we have the asymptotic expansion
		\[G^{\pm k}_{\square-z}\delta_0\sim\insum{m}\binom{k+m-1}{m}z^mR^V_\pm(2m+2k)(\cdot,0).\]
	\end{lemma}
	\begin{proof}
		Theorem \ref{PowExp} applied to $\square-z$ gives us the asymptotic expansion
		\[G^{\pm k}_{\square-z}\delta_0=\K(G^{\pm k}_{\square-z})(\cdot,0)\sim\insum{m}\binom{k+m-1}{m}V^ m(z)(\cdot,0)R^V_\pm(2m+2k)(\cdot,0),\]
		where $V^m(z)$ denotes the Hadamard coefficients for $\square-z$. The Hadamard coefficients of $\square$ are given by 
		\[V^k=\begin{cases}^ {1, \ \ k=0}_{0,\ \ k\neq0}\end{cases},\]
		as this makes all transport equations become $0=0$. Thus, by Proposition \ref{Vz}, the Hadamard coefficients for $\square-z$ are given by 
		\[V^k(z)=z^m.\]
		Inserting this, we obtain
		\[G^{\pm k}_{\square-z}\delta_0\sim\insum{m}\binom{k+m-1}{m}z^mR^V_\pm(2m+2k)(\cdot,0).\qedhere\]
	\end{proof}
	This carries over from Minkowski space to $M$ in a straightforward way:
	\begin{lemma}
		\label{Rzex}
		We have the asymptotic expansion
		\[R_\pm(z,2k)\sim \insum{m}\binom{k+m-1}{m}z^mR^{U}_\pm(2m+2k).\]
	\end{lemma}
	\begin{proof}
		As $U$ is convex, it is contractible and thus $TU$ is trivial. Let $\phi\colon U\times\R^d\rightarrow TU$ be an isometric trivialization that preserves time-orientation and denote by $\phi_x$ its restriction to the fibre over $x$, identifying $\{x\}\times\R^d$ with $\R^d$.
		As $\phi_x$ is a time-orientation preserving isometry, it preserves all objects defined only in terms of the metric and time-orientation, i.e.
		\[\phi_x^{-1*}(R^{\R^ d}_\pm(2m+2k))=R^{T_x M}_\pm(2m+2k)\]
		and
		\[\phi_x^{-1*}(G^{\pm k}_{\square_{\R^d}-z}\delta_0)=G^{\pm k}_{\square_{T_xM}-z}\delta_0.\]
		Let \[\Exp\colon \Dom(\Exp)\subseteq TU\rightarrow U\times U\]
		denote the exponential map with variable basepoint, i.e. $\Exp(v)=(x,\exp_x(v))$ for $v\in T_xU$.
		
		Let $n\in\N$ be arbitrary. For $N$ large enough such that
		\[G^{\pm k}_{\square_{\R^d}-z}\delta_0=\sum\limits_{m=0}^{N}\binom{k+m-1}{m}z^mR^{\R^ d}_\pm(2m+2k)(\cdot,0)+F\]
		with $F\in C^n(\R^d)$,
		we have
		\begin{align*}
		&R_\pm(z,2k)(\cdot,x)\\
		&=(\exp_x)^{-1*}(G^{\pm k}_{\square_{T_xM}-z}\delta_0)\\
		&=(\exp_x\circ\phi_x)^{-1*}(G^{\pm k}_{\square_{\R^d}-z}\delta_0)\\
		&=(\exp_x\circ\phi_x)^{-1*}\left(\sum\limits_{m=0}^{N}\binom{k+m-1}{m}z^mR^{\R^ d}_\pm(2m+2k)(\cdot,0)+F\right)\\
		&=\sum\limits_{m=0}^{N}\binom{k+m-1}{m}z^m(exp_x)^{-1*}R^{T_xM}_\pm(2m+2k)(\cdot,0)+F\circ\phi_x^{-1}\circ exp_x^{-1}\\
		&=\sum\limits_{m=0}^{N}\binom{k+m-1}{m}z^mR^{U}_\pm(2m+2k)(\cdot,x)+F\circ\phi^{-1}\circ \Exp^{-1}(\cdot,x).\\
		\end{align*}
		as $F\circ\phi^{-1}\circ \Exp^{-1}$ is $C^n$ (also as a function of $x$) and $n$ was arbitrary, we have the desired expansion.
	\end{proof}
	
	We now have everything in place to prove the main result of this section: If, in the asymptotic expansions for the Green's operators of $P$, we replace the standard Riesz distributions with the resolvent Riesz distributions for $z\in \C$, then we obtain asymptotic expansions for the Green's operators of $P-z$.
	\begin{thm}
		\label{Hadamexpz}
		If $U\subseteq M$ is GE, we have the asymptotic expansion
		\[\K(G^{\pm}_{P-z})|_{U\times U}\sim \insum{k}V^{k,U}R^U_\pm(z,2k+2).\]
		More generally,
		\[\K(G^{\pm}_{P-z}\null^m)|_{U\times U}\sim \insum{k}\binom{k+m-1}{k}V^{k,U}R^U_\pm(z,2k+2m).\]
	\end{thm}
	\begin{proof}
		Inserting our formula for the $z$-dependent Hadamard coefficients (Proposition \ref{Vz}) into the expansion for powers in Theorem \ref{PowExp}, applied with $P-z$ instead of $P$, we obtain
		\begin{align*}
		&\K(G^{\pm}_{P-z}\null^m)\\
		&\sim \insum{k}\binom{k+m-1}{k}V^{k}(z)R_\pm(2k+2m)\\
		&=\insum{k}\binom{k+m-1}{k}\sum\limits_{l+n=k}\binom{k}{n}z^nV^{l}R_\pm(2k+2m)\\
		&=\insum{l}\insum{n}\binom{n+l+m-1}{n+l}\binom{n+l}{n}z^nV^{l}R_\pm(2(n+m+l))\\
		&=\insum{l}\insum{n}\binom{n+l+m-1}{n}\binom{l+m-1}{l}z^nV^{l}R_\pm(2(n+m+l))\\
		&=\insum{l}\binom{l+m-1}{l}V^{l}\insum{n}\binom{n+l+m-1}{n}z^nR_\pm(2(n+m+l))\\
		&\sim\insum{l}\binom{l+m-1}{l}V^{l}R_\pm(z,2m+2l).\\
		\end{align*}
		Here we used the identity
		\[\binom{a}{b}\binom{b}{c}=\binom{a}{c}\binom{a-c}{b-c}\]
		and Lemma \ref{Rzex}.
		
		This proves the second claim, the first claim follows as the special case $m=1$.
	\end{proof}
\newpage
\printbibliography
\end{document}